\documentclass[12pt,reqno]{amsart}

\usepackage{amsmath, amsfonts, amsthm, amssymb, graphicx}
\usepackage{float,epsf,subfigure}

\theoremstyle{plain}
\newtheorem{theorem}{Theorem}[section]

\newtheorem{proposition}{Proposition}[section]

\theoremstyle{definition}

\theoremstyle{remark}
\newtheorem{remark}{Remark}[section]

\theoremstyle{example}

\numberwithin{equation}{section}

\newcommand{\e}{\varepsilon}

\renewcommand{\u}{{\bf u}}
\newcommand{\x}{{\bf x}}

\newcommand{\R}{{\mathbb R}}

\newcommand{\f}{{\bf f}}
\renewcommand{\v}{{\bf v}}

\newcommand{\ee}{{\bf e}}
\newcommand{\kk}{{\bf k}}
\newcommand{\w}{{\bf w}}
\newcommand{\y}{{\bf y}}
\newcommand{\varphib}{\boldsymbol{\varphi}}
\newcommand{\mmu}{{\boldsymbol{\mu}}}

\newcommand{\TP}{{\mathbb T_P}}

\def\be{\begin{equation}}
\def\ee{\end{equation}}
\def\bes{\begin{equation*}}
\def\ees{\end{equation*}}
\def\bc{\begin{cases}}
\def\ec{\end{cases}}

\begin{document}
\title[Kolmogorov's Theory and Inviscid Limit]
{Kolmogorov's Theory of Turbulence and Inviscid Limit of
the Navier-Stokes Equations in $\R^3$}

\author{Gui-Qiang Chen \and James Glimm}
\address{Gui-Qiang G. Chen, Mathematical Institute, University of Oxford,
         Oxford, OX1 3LB, UK;  School of Mathematical Sciences, Fudan University,
 Shanghai 200433, China; Department of Mathematics, Northwestern University,
         Evanston, IL 60208-2730, USA}
\email{chengq@maths.ox.ac.uk}
\address{James Glimm, Department of Applied Mathematics and Statistics,
         Stony Brook University, Stony Brook, NY 11794-3600, USA}
\email{glimm@ams.sunysb.edu}

\date{\today}

\maketitle

\begin{abstract}
We are concerned with the inviscid limit of the Navier-Stokes equations
to the Euler equations in $\R^3$.
We first observe that a pathwise Kolmogorov hypothesis implies
the uniform boundedness of the $\alpha^{th}$-order
fractional derivative of the velocity for some $\alpha>0$ in the space variables in $L^2$, which is
independent of the viscosity $\mu>0$.
Then it is shown that this key observation yields the $L^2$-equicontinuity in the time
and the uniform bound in $L^q$, for some $q>2$, of the velocity
independent of $\mu>0$.
These results lead to the strong convergence of solutions of the Navier-Stokes equations
to a solution of the Euler equations in $\R^3$.
We also consider passive scalars coupled to the incompressible
Navier-Stokes equations and, in this case, find the weak-star convergence
for the passive scalars with a limit in the form of a Young measure (pdf
depending on space and time).
Not only do we offer a framework for mathematical existence theories,
but also we offer a framework for the interpretation of numerical
solutions through the identification
of a function space in which convergence should
take place, with the bounds that are independent of $\mu>0$, that is in the high
Reynolds number limit.
\end{abstract}
\maketitle

\section{Introduction}

Consider the incompressible Navier-Stokes equations in $\R^3$:
\begin{equation}\label{eq:ns}
\left\{\begin{aligned}
&\partial_t\u+\nabla\cdot(\u\otimes \u)+\nabla p=\mu \Delta \u +\f,\\
&\nabla\cdot \u=0,
\end{aligned}\right.
\end{equation}
with Cauchy data:
\begin{equation}\label{1.2}
\u|_{t=0}=\u_0(\x),
\end{equation}
where $\u$ is the fluid velocity, $p$ is the pressure, $\mu>0$ is the viscosity,
$\nabla$ is the gradient with respect to the space variable $\x\in\R^3$,
$\u\otimes\u=(u_iu_j)$ is the $3\times 3$ matrix for $\u=(u_1, u_2, u_3)$,
and $\f=\f(t,\x)$ is a given external force.

The global existence theory for the Cauchy problem \eqref{eq:ns}--\eqref{1.2}
was first established by J. Leray  \cite{Leray1,Leray2,Leray3};
also see Hopf \cite{Hopf}, Temam \cite{Temam},
J.-L. Lions \cite{JLions}, P.-L. Lions \cite{PLions},
and  the references cited therein.
For clarity of presentation, we focus on periodic solutions with
period $\TP=[-P/2, P/2]^3\subset\R^3, P>0$,
that is,
$$
\u^\mu(t,\x+P{\bf e}_i)=\u^\mu(t,\x)
$$
with $({\bf e}_i)_{1\le i\le 3}$ the canonic basis in $\R^3$. Other cases
can be analyzed correspondingly.
We always assume that $\f\in L^2_{loc}(0, \infty; L^2(\TP))$ is periodic
in $\x$ with period $\TP$.

\begin{theorem}
Let $\u_0\in L^2(\TP)$ be periodic in $\x\in\R^3$ with period $\TP$.
Then, for any $T>0$, there exists a periodic weak solution $\u^\mu=\u^\mu(t,\x)$
with period $\TP$, along with a corresponding periodic pressure field $p^\mu(t,\x)$,
of \eqref{eq:ns}--\eqref{1.2}  such that the equations in \eqref{eq:ns}
hold in the sense of distributions in $\R_T^3:=[0, T)\times \R^3$,
and the following properties hold:
\begin{eqnarray}
&& \u^\mu\in L^2(0,T; H^1)\cap C([0, T]; L_w^2)\cap  C([0,T]; L^{s_1}),\\
&&\partial_t \u^\mu \in L^2(0,T; H^{-1})
+\big(L^{s_2}(0,T; W^{-1, \frac{3s_2}{3s_2-2}})\cap L^q(0,T; L^r)\big),\qquad\\
&&p^\mu\in L^2((0, T)\times \TP)+L^{s_2}(0,T; L^{\frac{3s_2}{3s_2-2}}),\\
&&\nabla p^\mu\in L^2(0,T; H^{-1})+L^q(0,T; L^r),
\end{eqnarray}
where $1\le s_1<2, 1\le s_2<\infty$, $1\le q<2$, and $r=\frac{3q}{2(2q-1)}$;
and, in addition,
\begin{equation}\label{energy-ineq}
\partial_t(\frac{1}{2}|\u^\mu|^2)+\nabla\cdot\big(\u^\mu (\frac{1}{2}|\u^\mu|^2+p^\mu)\big)
+\mu |\nabla\u^\mu|^2-\mu \Delta(\frac{|\u^\mu|^2}{2})\le \f\cdot \u^\mu
\end{equation}
in the sense of distributions in $\R_T^3$.
\end{theorem}

In Theorem 1.1, $\v\in C([0,T]; L_w^2(\TP))$ means that $\v\in L^\infty(0, T; L^2(\TP))$ and
$\v$ is continuous in $t$ with values in $L^2(\TP)$ endowed with the weak topology.
Some further a priori estimates and properties of solutions to the Navier-Stokes equations \eqref{eq:ns}
can be found in \cite{Const90,FMRT,PLions} and the references cited therein.

\medskip
In contrast, less is known
regarding the existence theory for the Euler equations
(defined with $\mu = 0$ in (\ref{eq:ns})). For the compressible case,
the analysis of
\cite{GW,Neustupa}
gives convergence subsequences, but to a
very weak limit as a measure-valued vector function. Moreover, this limit
is not shown to satisfy the original equations, in that the interchange
of limits with nonlinear terms in the equations is not justified in this
analysis. On this basis, we state that the
existence of solutions for the Euler equations in $\R^3$ is open as is the
convergence of the inviscid limit from the Navier-Stokes to the Euler equations.
On the other hand, the Euler equations are fundamental for turbulence; see Constantin
\cite{Const} and the references cited therein.

\medskip
The purpose of this paper is to establish a framework for the existence
theory for the Euler equations. We introduce a physically well accepted
pathwise hypothesis by Kolmogorov \cite{Kolmogorov1,Kolmogorov2}, \textit{Assumption (K41)}.
One of our main contributions in this paper is our key observation that
\textit{Assumption (K41)}, even a weaker version \textit{Assumption (K41w)},
assures sufficient regularity of the solutions to the Navier-Stokes
equations with an external force that convergence through a subsequence to solutions of the
Euler equations is guaranteed.

As is well-known, there are two types of turbulence: driven turbulence by a forcing function
and transient turbulence by (strong) initial conditions. The same flow
(in a turbulent wind tunnel or turbulent flow in a pipe) could be
of either type depending on how the system is modeled:
If the pipe is considered in isolation, the flow might be forced;
if the force is a flow connection from a reservoir to the wind tunnel
or pipe, and the reservoir is part of the model,
then the turbulence arises from initial conditions, and the turbulence
will die out when the reservoir is exhausted and no longer drives the
flow. Thus, the distinction between the two (transient and driven turbulence)
is to some extent a matter of points of view and of modeling convenience.
In this paper, our framework is focused on transient turbulence, though the
forced turbulence is also included.
For some recent developments in the mathematical study of energy dissipation
in body-forced turbulence, see Constantin-Doering \cite{CD},
Doering-Foias \cite{DF}, and the references cited therein.

Whenever a mathematically precise formulation of a physically precise idea
is attempted, ambiguities may arise. In the present case, the assumptions
concern the rate $\epsilon$ of dissipation of kinetic energy, with
a fundamental definition which depends on the viscosity, as the ultimate source
of energy dissipation.  The essence of \textit{Assumption (K41)},
as stated in a physical language, is that in the inertial range,
the energy is transferred from larger to smaller length scales (eddies)
in a manner which is independent of viscosity, because all aspects
of the inertial range are assumed to be independent of viscosity.
The energy dissipation occurs only at smaller scales, i.e.,
below the Kolmogorov scale, and even there, it is limited by the energy
arriving at these scales through the energy cascade.
Given that the energy dissipation rate is supposed to be independent of viscosity,
the question remains as to which statistical ensemble or in what norm or function space
to express this property.
Here the physical literature is not a good guide,
and we introduce the assumptions sufficient to allow our proofs to go forward.
The distinction among time averages, spatial averages at a fixed time,
space-time averages
and ensemble averages relates to the well known ergodic hypothesis
and is out of the scope of the present paper.

\medskip
Mathematically, our result has the
status of an informed conjecture and mathematically rigorous
consequences of this conjecture. Numerical analysts may find
the framework useful, in view of the many difficulties involved in
assessing convergence of numerical simulations of
turbulent and turbulent mixing flows.
We expect that many physicists will probably
accept the conclusions as being correct,
even if unproven mathematically. There has
been some discussion regarding the exponent $5/3$ which occurs in
\textit{Assumption (K41)}.
We note that the main results (if not the detailed estimates) are
not sensitive to this specific number, and corrections (as conventionally
understood) to it due to intermittency should not affect our result.
In fact, our rigorous argument works for an even weaker version, \textit{Assumption (K41w)},
for any $\beta>0$.

Not only do we offer a framework for mathematical existence theories,
but also we offer a framework for the interpretation of numerical solutions
of (\ref{eq:ns}). Only for very modest problems and with the largest
computers can converged solutions of (\ref{eq:ns}) be achieved. These
solutions are called direct numerical solutions (DNS). For most solutions
of interest to science or engineering, the large eddy simulations
(LES) or Reynolds Averaged Navier-Stokes (RANS) simulations are required.
We discuss here the more accurate LES methodology. Briefly, it
employs a numerical grid which will resolve some but not all of the
turbulent eddies. The smallest of those, below the level of the grid
spacing, are not resolved. However, either (a) the equations in (\ref{eq:ns})
are modified with additional subgrid scale (SGS) terms to model the
influence of the unresolved scales on those that are being computed
or (b) the numerical algorithm is modified in some manner to
accomplish this effect in some other way. The present article contributes
to this analysis through the introduction of a function space in
which convergence should take place, with bounds that are independent
of $\mu$, that is in the high Reynolds number limit.

Because of the common occurrence of high Reynolds numbers in flows of
practical and scientific interest and the need to perform LES simulations
to achieve scientific understanding and engineering designs, we observe
that existence theories for the Euler equations are relevant to the
mathematical theories of numerical analysis.

%%%%%%%%%%%%%%%%%
\medskip
\section{The Kolmogorov  Hypotheses (1941)}

\medskip
By the definition of weak solutions, we know that, for any $T>0$,
there exists $C_T>0$ independent of $\mu$ such that
\begin{eqnarray}\label{energy-estimate}
&&\|\u^\mu-\bar{\u}\|_{L^\infty (0,T; L^2(\TP))}^2
+\|\sqrt{\mu}\nabla \u^\mu\|_{L^2([0,T)\times\TP)}^2\nonumber\\
&&\le C_T\big(\|\u_0-\bar{\u}\|_{L^2(\TP)}^2 +\|\f-\bar{\f}\|_{L^2([0,T)\times\TP)}^2\big),
\label{energy-estimate}
\end{eqnarray}
where
$$
\bar{\u}(t)=\frac{1}{|\TP|}\int_\TP\u^\mu(t,\x)d\x=\frac{1}{|\TP|}\int_\TP\u_0(\x)d\x +\int_0^t\bar{\f}(s)ds
$$
and
$$
\bar{\f}(t)=\frac{1}{|\TP|}\int_\TP\f(t,\x)d\x
$$
are independent of $\mu$.
Without loss of generality, we assume
that the mean velocity and pressure are zero, and interpret $\u^\mu$
as the fluctuating velocity.
Then the total energy $\mathcal{E}(t)$ per unit mass at time $t$
for isotropic turbulence is:
\begin{equation}
\mathcal{E}(t)=\frac{1}{2\TP}\int_{\TP}|\u^\mu(t,\x)|^2d\x
=\sum_{k\ge 0} E(t,k)
=\sum_{k\ge 0} 4\pi q(t,k)k^2.
\end{equation}
Here $E(t,k), k=|\kk|$, is the energy wavenumber spectrum, $q(t,k)$ can be interpreted as
the density of contributions in wavenumber space
to the total energy, which is sometimes called the spectral density,
and $\kk=(k_1, k_2, k_3)=\frac{2\pi}{P}(n_1,n_2,n_3)\in \R^3$, with $n_j=0, \pm 1, \pm2, \cdots$, and $j=1,2,3$,
is the discrete wavevector in the Fourier transform:
$$
\hat{\u}(t,\kk)=\frac{1}{|\TP|}\int_{\TP}\u(t,\x)e^{-i\kk\cdot\x}d\x
$$
of
the velocity $\u(t,\x)$ in the $\x$-variable.
Then
$$
\u(t,\x)=\sum_{\kk}\hat{\u}(t,\kk)e^{i\kk\cdot\x}
$$

\medskip
Kolmogorov's two assumptions in his description of isotropic turbulence
in Kolmogorov \cite{Kolmogorov1,Kolmogorov2} (also see  McComb \cite{McComb})
are:

\medskip
\begin{enumerate}
\item[(i)] At sufficiently high wavenumbers, the energy spectrum $E(t,k)$, can
depend only on the fluid viscosity $\mu$, the dissipation rate $\varepsilon$,
and the wavenumber $k$ itself.

\item[(ii)] $E(t,k)$ should become \textit{independent
of the viscosity} as the Reynolds number tends to infinity:

\begin{equation}\label{2.2a}
E(t, k)\approx  \alpha \varepsilon^{2/3}k^{-5/3}
\end{equation}
in the limit of infinite Reynolds number, where $\alpha$ may depend on $t$, but is independent
of $k$ and $\varepsilon$.
\end{enumerate}

\smallskip
For general turbulence, the energy wavenumber spectrum $E(t,k)$ in \eqref{2.2a}
may be replaced by $E(t, \kk)=E(t, k, \phi, \theta)$ in the spherical
coordinates $(k, \phi, \theta)$, $0\le \phi\le \pi$,
$0\le \theta \le 2\pi$, in the $\kk$-space,
but it should be in the same asymptotics as in \eqref{2.2a}
for sufficiently high wavenumber $k=|\kk|$.

\medskip
For the remainder of this paper, we assume these two hypotheses,
which we interpret in mathematical terms as a pathwise Kolmogorov hypothesis:

\bigskip
{\bf Assumption (K41):} {\it For any $T>0$, there exists $C_T>0$ and $k_*$
(sufficiently large) depending on $\u_0$ and $\f$ but
independent of the viscosity $\mu$ such that,
for $k=|\kk|\ge k_*$,
\begin{equation} \label{assumption}
\int_0^T E(t, \kk)dt \le C_T k^{-5/3}.
\end{equation}
}

\medskip
We should point out that the constant $C_T=C_T(\u_0, \f)$ depends on the
maximum time $T$, the initial data $\u_0$, and the external force $\f$,
but is independent of the viscosity $\mu>0$.
For given the maximum time $T$, the dependence on $\u_0$ and $\f$ of the constant $C_T$ is from the control
of the upper bound of $\int_0^T \e^{2/3}(s)ds$, with $\e=\e(t)$ being the rate of energy dissipation,
by the initial data and the external force, more specifically by $\|\u_0\|_{L^2(\TP)}$ and $\|\f\|_{L^2([0,T)\times\TP)}$.
The dimensionless version of the bound follows from \eqref{2.2a}.

\bigskip
For our analysis, the following weaker version of {\it Assumption (K41)} is sufficient:

\medskip
{\bf Assumption (K41w):}
{\it For any $T>0$, there exists $C_T=C_T>0$ and $k_*$
(sufficiently large) depending on $\u_0$ and $\f$ but
independent of the viscosity $\mu$ such that,
for $k=|\kk|\ge k_*$,
\begin{equation} \label{assumption}
\sup_{k\ge k_*}\Big(|\kk|^{3+\beta}\int_0^T|\hat{\u}(t, \kk)|^2dt\Big)
\le C_T \qquad\quad\mbox{for some $\beta>0$}.
\end{equation}
}

\medskip
A mathematical proof of \textit{Assumption (K41)}
may well depend on developing a mathematical
version of the renormalization group. The renormalization group
has proved to be very powerful in theoretical physics calculations.
The basic idea is to integrate differentially a segment of the
problem, say from $\kk$ to $\kk - d\kk$, following by a rescaling of all
variables, so that all variables are redefined to be integrated
through $\kk$ only. The integration is quite a messy operation,
but in the rescaling, it is important to examine the scaling dimensions
of all terms. Most of them get smaller under rescaling and are called inessential.
The few that are preserve their magnitude are called essential and
serve to define the key parameters of the problem. These must be reset
(renormalized) to their observed (interaction) values after the integration.
Since the method is generally applied to self-similar problems, the
result of integration does not change the form of the equations.
That is, it looks the same at $\kk$ as at $\kk - d\kk$ rescaled back to $\kk$.
Thus, one has a kind of group operation, if the steps are discrete, and
a differential equation, if the steps, as the present notation suggests,
are infinitesimal. In either case, the desired solution is a fixed point.

\medskip
\section{$L^2$--Equicontinuity of the Velocity in the Space Variables,
Independent of the Viscosity}

\medskip
In this section we show that the pathwise Kolmogorov hypothesis,
\textit{Assumption (K41w)}, implies
a uniform bound of the velocity $\u^\mu(t,\x)$ in $L^2(0,T; H^\alpha(\TP))$ for any
$\alpha\in (0, \beta/2)$, especially the uniform equicontinuity
of the velocity $\u^\mu(t,\x)$ in the space variables in $L^2([0,T)\times \TP)$,
independent of $\mu>0$.

\begin{proposition}\label{prop:3.1}
Under {\rm Assumption (K41w)}, for any $T\in (0, \infty)$,
there exists $M_T>0$ depending on $\u_0, \f$, and $T>0$, but independent of $\mu>0$,
such that
\begin{equation}\label{3.1}
\|\u^\mu\|_{L^2(0,T; H^{\alpha}(\TP))}\le M_T<\infty,
\end{equation}
where $\alpha \in (0, \beta/2)$.
\end{proposition}

\begin{proof} Using the definition of fractional derivatives via the Fourier transform,
the Parseval identity, and Assumption (K41w), i.e., \eqref{assumption},
we have
\begin{eqnarray*}
&&\int_0^T\int_{\TP} |D^\alpha_\x \u^\mu(t,\x)|^2 d\x dt\\
&&\le C_1\int_0^T \Big(\sum_{\kk} |\kk|^{2\alpha}|\hat{\u^\mu}(t,\kk)|^2 \Big) dt\\
&&=C_1\int_0^T \Big(\sum_{0\le|\kk|\le k_*} |\kk|^{2\alpha}|\hat{\u^\mu}(t,\kk)|^2 \Big) dt
   +C_1\int_0^T \Big(\sum_{|\kk|>k_*} |\kk|^{2\alpha}|\hat{\u^\mu}(t,\kk)|^2 \Big) dt\\
&&\le C_1 k_*^{2\alpha}\int_0^T \Big(\sum_{0\le|\kk|\le k_*}|\hat{\u^\mu}(t,\kk)|^2 \Big) dt
   +C_2\sum_{|\kk|\ge k_*} |\kk|^{2\alpha-3-\beta}\\
&&\le C_1 k_*^{2\alpha}\int_0^T \Big(\sum_{\kk}|\hat{\u^\mu}(t,\kk)|^2 \Big) dt
   +C_3\sum_{k\ge k_*} k^{2\alpha-1-\beta}\\
&&\le C_1k_*^{2\alpha}\int_0^T\int_{\TP}|\u^\mu(t,\x)|^2d\x dt
   +C_4 k_*^{2\alpha-\beta}\\
&&\le C_5k_*^{2\alpha}T\Big(\int_{\TP}|\u_0(\x)|^2d\x+\int_0^T\int_{\TP}|\f(t,\x)|^2d\x dt\Big)
+C_4 k_*^{2\alpha-\beta}\\
&&\le M(T, k_*, \alpha)<\infty,
\end{eqnarray*}
since  $\alpha< \beta/2$, which $C_j, j=1, \cdots, 5$, are the constants independent of $\mu$.
This completes the proof.
\end{proof}

Proposition 3.1 directly yields the uniform equicontinuity of $\u^\mu(t,\x)$
in $\x$ in $L^2([0, T)\times\TP)$ independent of $\mu>0$.

\medskip
\section{$L^2$--Equicontinuity of the Velocity in the Time-Variable,
Independent of the Viscosity}

In this section, we show that Proposition 3.1 implies the uniform equicontinuity
of the velocity in the time variable $t>0$ in $L^2$, independent of $\mu>0$.

\begin{proposition}\label{prop:4.1}
For any $T>0$, there exists $M_T>0$ depending on $\u_0, \f$, and $T>0$, but
independent of $\mu>0$, such that,
for all small $\triangle t>0$,
\begin{equation}\label{4.1}
\int_0^{T-\triangle t}\int_{\TP} |\u^\mu (t+\triangle t, \x)-\u^\mu(t,\x)|^2d\x dt
\le M_T(\triangle t)^{\frac{2\alpha}{5+2\alpha}}\to 0 \qquad \mbox{as}\,\,
\triangle t\to 0.
\end{equation}
\end{proposition}

\begin{proof} For simplicity, we drop the superscript $\mu>0$ of $\u^\mu$ in
the proof.

\medskip
Fix $\triangle t>0$. For $t\in [0, T-\triangle t]$,
set
$$
\w(t,\cdot)=\u(t+\triangle t, \cdot)-\u(t, \cdot).
$$
Then, for any $\varphib(t,\x)\in C^\infty([0,T)\times\TP)$ that is periodic in $\x\in\R^3$ with
period $\TP$, we have
\begin{eqnarray}
&&\int_{\TP}\w(t,\x)\cdot \varphib(t,\x)\, d\x \nonumber\\
&&=\int_{t}^{t+\triangle t}\int_{\TP} \partial_s\u(s,\x)\cdot \varphib(t,\x)\, d\x ds\nonumber\\
&&=\int_t^{t+\triangle t}\int_\TP
 (\u\otimes\u)(s, \x): \nabla\varphib(t,\x) \, d\x ds \nonumber\\
&&\quad +\int_t^{t+\triangle t}\int_{\TP} p(s,\x) \nabla\cdot\varphib(t,\x)\, d\x
ds\nonumber\\
&& \quad -\mu \int_t^{t+\triangle t}\int_{\TP} \nabla
\u(s,\x) :\nabla\varphib(t,\x)\,  d\x ds \nonumber\\
&&\quad +\int_t^{t+\Delta t}\int_\TP \f(s,\x)\cdot \varphib(t,\x)d\x ds,
\label{4.2}
\end{eqnarray}
where $\nabla\varphib=(\partial_{x_j}\varphi_i)$ is the $3\times 3$ matrix,
and $A:B$ is the matrix product $\sum_{i,j} a_{ij}b_{ij}$ for $A=(a_{ij})$
and $B=(b_{ij})$.

By approximation, equality \eqref{4.2} still holds
for
$$
\varphib\in L^\infty(0,T; H^1(\TP))\cap C([0,T]; L^2(\TP))\cap L^\infty([0,T)\times\TP).
$$
Choose
$$
\varphib=\varphib^\delta(t,\x):=(j_\delta*\w)(t,\x)=\int
j_\delta(\y)\w(t,\x-\y)\, d\y\in\R^3
$$
for $\delta>0$, which is periodic in $\x\in\R^3$ with period $\TP$, where $j_\delta(\x)=\frac{1}{\delta^3}j(\frac{\x}{\delta})\ge 0$
is a standard mollifier
with $j\in C_0^\infty(\R^3)$ and $\int j(\x)d\x=1$.
Then
\begin{equation}\label{4.3a}
\nabla\cdot \varphib^\delta(t,\x)=0,
\end{equation}
since $\nabla\cdot \w(t,\x)=0$.

Integrating \eqref{4.2} in $t$ over $[0, T-\triangle t)$ with
$\varphib=\varphib^\delta(t,\x)$ and using \eqref{4.3a}, we have
\begin{eqnarray}
&&\int_0^{T-\triangle t}\int_{\TP} |\w(t,\x)|^2 d\x dt
\nonumber\\
&&=\int_0^{T-\triangle t}\int_t^{t+\triangle t}\int_{\TP}
(\u\otimes\u)(s,\x): \nabla\varphib^\delta (t,\x) \, d\x ds dt\nonumber\\
&&\quad -\mu \int_0^{T-\triangle t}\int_{\TP}\int_t^{t+\Delta t}
\nabla \u(s,\x): \nabla\varphib^\delta(t,\x)\, ds d\x dt\nonumber\\
&&\quad + \int_0^{T-\triangle t}\int_{\TP} \w(t,\x)\cdot \big(\w(t,\x)
-\varphib^\delta(t,\x)\big)\, d\x dt \nonumber\\
&&\quad +\int_0^{T-\Delta t}\int_t^{t+\Delta t}\int_\TP
\f(s,\x)\cdot \varphib^\delta(t,\x)d\x ds dt\nonumber\\
&&=:J_1^\delta +J_2^\delta+J_3^\delta +J_4^\delta.
\label{4.2a}
\end{eqnarray}

Notice that, for any $\x\in\TP$,
\begin{eqnarray*}
|\nabla\varphib^\delta(t,\x)|
&\le& \frac{1}{\delta}
\Big|\int_{|\x-\y|\le \delta} j_\delta'(\x-\y)\w(t,\y)\, d\y\Big|\\
&\le& \frac{C}{\delta^{5/2}}\big(\int |j'(\y)|^2 d\y\big)^{1/2}\|\w(t,\cdot)\|_{L^2(\TP)}\\
&\le& \frac{C}{\delta^{5/2}}.
\end{eqnarray*}
Here and hereafter, we use $C>0$ as a universal constant independent
of $\mu>0$.

Then
\begin{equation}\label{4.4}
|J_1^\delta|\le \frac{C T\triangle t}{\delta^{5/2}}\|\u(t,\cdot)\|_{L^2(\TP)}^2
\le \frac{C T\triangle t}{\delta^{5/2}}\|\u_0\|_{L^2(\TP)}^2\le \frac{C\triangle t}{\delta^{5/2}}.
\end{equation}

Furthermore, we have
\begin{eqnarray}
|J_2^\delta|&\le& \int_0^{T-\triangle t}\int_{\TP}\int_t^{t+\triangle t}
   \mu |\nabla \u(s,\x)||\nabla \varphib^\delta(t,\x)|\, ds d\x
dt\nonumber\\
&\le& \int_0^{T-\triangle t}\int_{\TP} \Big(\int_t^{t+\Delta t}
   \sqrt{\mu} |\nabla \u(s,\x)| ds\Big)
      \sqrt{\mu}|\nabla \varphib^\delta(t,\x)|\, d\x dt\nonumber\\
&\le&\Big(\int_0^{T-\triangle t}\int_{\TP}
 \big(\int_t^{t+\triangle t}\sqrt{\mu}|\nabla \u(s,\x)| ds\big)^2 d\x
dt\Big)^{1/2}\nonumber\\
&&\times
 \Big(\int_0^{T-\triangle
t}\int_{\TP}\mu|\nabla\varphib^\delta(t,\x)|^2d\x dt \Big)^{1/2}\nonumber\\
&\le&
(\triangle t)^{1/2}\Big(\int_0^{T-\triangle t}\int_{\TP}
 \int_t^{t+\triangle t}\mu|\nabla \u(s,\x)|^2 ds d\x
dt\Big)^{1/2}\nonumber\\
&&\times
\Big(\int_0^{T-\triangle t}
\int_{\TP}\mu \big(\int_{|\y|\le\delta} j_\delta(\y)|\nabla \w(t, \x-\y)| d\y\big)^2d\x dt \Big)^{1/2}\nonumber\\
&\le& C \triangle t \Big(\mu \int_{\R^3} |j_\delta(\y)|^2d\y
           \int_{|\y|\le \delta}\int_0^{T-\triangle t}\int_{\TP} \mu |\nabla
\w(t,\x-\y)|^2d\x dt d\y\Big)^{1/2} \nonumber\\
&\le& C \triangle t, \label{4.5}
\end{eqnarray}
where we have used
$$
\|\sqrt{\mu}\nabla\u\|_{L^2([0, T)\times\TP)}^2\le M_T
$$
from \eqref{energy-ineq} for the weak solutions in Theorem 1.1 with $M_T>0$ independent of $\mu$.

On the other hand, for
$$
J_3^\delta:=\int_0^{T-\triangle t}\int_{\TP}
\w(t,\x)\big(\w(t,\x)-\varphib^\delta(t,\x)\big)d\x dt,
$$
we find
\begin{eqnarray}
|J_3^\delta|&=&\Big|\int_0^{T-\triangle t}\int_{\TP} \w(t,\x)\Big(\int_{\R^3}
j_\delta(\x-\y)(\w(t,\x)-\w(t,\y))d\y\Big) d\x dt\Big|
\nonumber\\
&\le& \int_0^{T-\triangle t}\int_{\TP}\int_{|\y|\le 1} j(\y) |\w(t,\x)| |\w(t,
\x)-\w(t,\x-\delta \y)|\, d\y d\x dt
\nonumber\\
&\le& \int_{|\y|\le 1} j(\y) \Big(\int_0^T\int_{\TP} |\w(t,\x)| |\w(t,
\x)-\w(t,\x-\delta \y)|\, d\x dt\Big) d\y dt\nonumber\\
&\le& C\int_{|\y|\le 1}  j(\y)
  \Big(\int_0^T\int_{\TP} |\w(t, \x)-\w(t,\x-\delta \y)|^2d\x dt \Big)^{1/2} d\y
dt\nonumber\\
&\le& C\delta^\alpha \int_{\R^3} j(\y)|\y|^\alpha d\y \nonumber\\
&\le&  C_3 \delta^\alpha.
\label{4.6}
\end{eqnarray}
Here we have used the fact that $\|\w(t,\cdot)\|_{L^2(\TP)}\le C$ and
\begin{eqnarray*}
&&\int_0^T\int_{\TP} |\w(t, \x)-\w(t,\x-\delta \y)|^2d\x dt\\
&&=\int_0^T\Big(\sum_{\kk} |\hat{\w}(t,\kk)|^2 (1-e^{i\delta\kk\cdot\y})^2\Big) dt\\
&&\le C \delta^{2\alpha} |\y|^{2\alpha}\int_0^T\Big(\sum_{\kk}|\kk|^{2\alpha}|\hat{\w}(t,\kk)|^2\Big) dt\\
&&\le C \delta^{2\alpha} |\y|^{2\alpha}\int_0^T\int_{\TP} |D_\x^\alpha(t,\x)|^2d\x dt\\
&&\le  C \delta^{2\alpha} |\y|^{2\alpha}.
\end{eqnarray*}

Moreover, we have
\begin{eqnarray}
|J_4^\delta|&=& C\Delta t \|\f\|_{L^2([0, T)\times\TP)}
           \|\varphib^\delta\|_{L^2((0, T-\Delta t)\times\TP)}\nonumber \\
 &\le& C\Delta t \|\w\|_{L^2((0, T-\Delta t)\times\TP)}\nonumber\\
 &\le&  C\Delta t \|\u\|_{L^2((0, T)\times\TP)}\nonumber\\
 &\le& C_4\Delta t.  \label{4.5a}
\end{eqnarray}

Combining \eqref{4.2a}--\eqref{4.6} with \eqref{4.5a}, we have
\begin{eqnarray*}
\int_0^{T-\triangle t}\int_{\TP} |\w(t,\x)|^2 d\x dt
&\le& \inf_{\delta>0} \{C_1\frac{\triangle t}{\delta^{5/2}}+(C_2+C_4) \triangle t
                      +C_3 \delta^\alpha\}\nonumber\\
&\le& \inf_{\delta>0} \{C_5\frac{ \triangle t}{\delta^{5/2}} +C_3
\delta^\alpha\}.
\end{eqnarray*}
Choose
$$
\delta=\Big(\frac{5C_5}{2\alpha C_3}\Big)^{\frac{2}{5+2\alpha}}
(\triangle t)^{\frac{2}{5+2\alpha}}.
$$
We conclude
\begin{equation}
\int_0^{T-\triangle t}\int_{\TP} |\w(t,\x)|^2 d\x dt
\le C (\triangle t)^{\frac{2\alpha}{5+ 2\alpha}}.
\end{equation}
This completes the proof.
\end{proof}

As a direct corollary, we have

\begin{proposition}\label{prop:4.2}
Under {\rm Assumption (K41w)}, for any $T\in (0, \infty)$,
there exists $M_T>0$ depending on $\u_0, \f$, and $T>0$, but independent of $\mu>0$,
such that
\begin{equation}\label{4.8}
\int_0^T\int_{\TP} |D^\alpha \u^\mu(t,\x)|^2 dt d\x\le M_T<\infty,
\end{equation}
for some $\alpha >0$.
\end{proposition}

\medskip

Combining Proposition 3.1 with Proposition 4.2, we have
\begin{equation}\label{4.9}
\|\u^\mu\|_{H^\alpha([0, T)\times\TP)}\le M_T<\infty
\end{equation}
for some $\alpha>0$.
Then, by the Sobolev imbedding theorem, we have

\begin{proposition}\label{prop:4.3}
Under {\rm Assumption (K41w)},  for any $T\in (0, \infty)$,
there exists $M_T>0$ depending on $\u_0, \f$, and $T>0$, but independent of $\mu>0$,
such that
\begin{equation}\label{3.2}
\|\u^\mu\|_{L^2\cap L^q([0, T)\times\TP)}\le M_T<\infty
\end{equation}
for some $q>2$.
\end{proposition}

\medskip

\section{Inviscid Limit}
In this section we show the inviscid limit from the Navier-Stokes to the Euler equations
under \textit{Assumption (K41w)}.

\begin{theorem}\label{thm:5.1}
The pathwise Kolmogorov hypothesis, {\rm Assumption (K41w)}, implies the strong compactness in
$L^2\cap L^q([0, T)\times\TP)$, for some $q>2$,
of the solutions $\u^\mu(t,\x)$ of the Navier-Stoke equations in $\R^3$ when
the viscosity $\mu$ tends to zero. That is,
there exist a subsequence (still denoted) $\u^\mu(t,\x)$
and a function $\u\in L^2\cap L^{q}([0,T)\times\TP)$
with a corresponding pressure function $p$ such that
$$
\u^\mu (t, \x)\to \u(t,\x) \qquad \mbox{a.e. as } \, \mu\to 0,
$$
and $\u(t,\x)$ is a weak solution of the
incompressible Euler equations with Cauchy data $\u_0(\x)$
along with the corresponding pressure $p$.
Furthermore,
\begin{equation}\label{energy-ineq-b}
\int_{\TP} |\u(t,\x)|^2d\x \le \int_{\TP}|\u_0(\x)|^2d\x+2\int_0^t\int_\TP\u(s,\x)\cdot\f(s,\x) d\x ds.
\end{equation}
\end{theorem}

\medskip

\begin{proof}
 Propositions 3.1 and 4.1 imply the $L^2$--equicontinuity of $\u^\mu(t,\x)$
in $(t,\x)\in [0, T)\times\TP$, independent of $\mu>0$.
This yields that there exists a subsequence (still denoted) $\u^\mu(t,\x)$
and a function $\u(t,\x)\in L^2$ such that
$$
\u^\mu(t,\x)\to \u(t,\x)      \qquad \mbox{in}\,\, L^2 \,\,\mbox{as $\mu\to 0$},
$$
which implies that
$$
\u^\mu(t,\x)\to \u(t,\x)      \qquad \mbox{a.e.}\,\,\mbox{as $\mu\to 0$},
$$
\begin{equation}\label{5.1}
(\u^\mu\otimes \u^\mu)(t,\x)\to (\u\otimes\u)(t,\x)      \qquad \mbox{a.e.}\,\,\mbox{as $\mu\to 0$},
\end{equation}
and
\begin{equation}\label{5.2}
\nabla\cdot \u(t,\x)=0
\end{equation}
in the sense of distributions in $\R_T^3$.

\medskip

{}From Proposition \ref{prop:4.3}, we have
$$
\u\in L^2\cap L^q([0, T)\times\TP),
$$
and
$$
\u^\mu(t,\x)\to \u(t,\x)
 \qquad \mbox{in}\,\, L^2\cap L^q([0, T)\times\TP) \,\,\mbox{as $\mu\to 0$}
$$
for some $q>2$.

\medskip
Furthermore, for any $\varphib\in C_0^\infty(\R_T^3)$ with $\nabla\cdot \varphib=0$,
we multiply $\varphib$ both sides \eqref{eq:ns} and integrate over $\R_T^3$ to
obtain
\begin{eqnarray*}
&&\int_0^T\int_{\R^3}\big(\u^\mu \varphib_t + (\u^\mu\otimes\u^\mu):\nabla \varphib
+\u^\mu\cdot\f\, \varphib\big)\, d\x dt
+\int_{\R^3} \u_0(\x)\varphib(0,\x)\, d\x\\
&&=-\mu \int_0^T\int_{\R^3} \u^\mu\cdot \triangle \varphib\, d\x dt.
\end{eqnarray*}
Letting $\mu\to 0$, we conclude that
\begin{equation}\label{5.4}
\int_0^T\int_{\R^3}\big(\u \varphib_t + (\u\otimes\u):\nabla \varphib
+(\u\cdot\f)\, \varphib\big)\, d\x dt
+\int_{\R^3} \u_0(\x)\varphib(0,\x)\, d\x=0.
\end{equation}
Combining \eqref{5.4} with \eqref{5.1} yields that $\u(t,\x)$ is a weak
solution of the Euler equations \eqref{eq:ns} ($\mu=0$) with Cauchy
data \eqref{1.2} along with the corresponding pressure $p$.

\medskip
Integrating \eqref{energy-ineq} over $[0, T)\times\TP$ yields
$$
\int_{\TP} |\u^\mu(t,\x)|^2 d\x
\le \int_{\TP} |\u_0(\x)|^2 d\x +2\int_0^t\int_\TP \u^\mu(s,\x)\cdot\f(s,\x) d\x ds.
$$
Letting $\mu\to 0$, we have
$$
\int_{\TP} |\u(t,\x)|^2 d\x
\le \int_{\TP} |\u_0(\x)|^2 d\x +2\int_0^t\int_\TP \u(s,\x)\cdot\f(s,\x) d\x ds.
$$
This completes the proof.
\end{proof}

\begin{remark}
The question whether the energy inequality in \eqref{energy-ineq-b}
becomes identical is exactly Onsager's conjecture \cite{Onsager}.
It states that weak solutions of the Euler equations for incompressible fluids
in $\R^3$ conserve energy only if they have a certain minimal smoothness of the
order of $1/3$ fractional derivatives in $\x\in\R^3$ and that they dissipate
energy if they are rougher.
Indeed, in Cheskidov-Constantin-Friedlander-Shvydkoy \cite{CCFS},
it is proved that energy is conserved when $\u(t,\x)$ is in the
Besov space of tempered distributions  $B^{1/3}_{2,c(\mathbb{N})}$.
This is a space in which the H\"{o}lder exponent is exactly $1/3$,
but the slightly better regularity is encoded in the summability condition.
We show in Proposition 3.1 that {\em Assumption (K41w)} implies
the uniform boundedness in the norm of $L^2(0, T; H^\alpha(\TP))$ for some $\alpha>0$
for the solutions to the Navier-Stokes
equations, independent of the viscosity $\mu>0$.
{}From the point of view of physics, strengthening the Kolmogorov
exponent $5/3$  (i.e. $\beta=2/3$), as would be needed to obtain a stronger $\alpha = 1/3$
bound, may result from some theory of intermittency.
\end{remark}

\section{Incompressible Navier-Stokes Equations with Passive Scalar Fields}

Consider the incompressible Navier-Stokes equations with passive scalar fields in $\R^3$:
\begin{equation}\label{6.1}
\left\{\begin{aligned}
&\partial_t\u+\nabla\cdot(\u\otimes \u)+\nabla p=\mu_0 \Delta \u +\f,\\
&\partial_t \chi_i +\nabla\cdot (\chi_i\u)=\mu_i\Delta \chi_i,  \quad i=1, \cdots, I,\\
&\nabla\cdot \u=0,
\end{aligned}\right.
\end{equation}
with Cauchy functions:
\begin{equation}\label{6.2}
(\u, \chi_1,\cdots, \chi_I)|_{t=0}=(\u_0, \chi_{10}, \cdots, \chi_{I0})(\x)
\end{equation}
that are periodic in $\x\in\R^3$ with period $\TP$,
where $\chi_{i0}(\x)\ge 0$ with $\sum_{i}\chi_{i0}(\x)=1$,
and $\f\in L^2_{loc}(0,\infty; L^2(\TP))$ is a given external force.

\smallskip
These equations are interpreted as describing concentrations $\chi_i$, $i < I$,
of minority species, such as dilute chemical reagents in a majority carrier
species $\chi_I$, in an approximation that the minority species do not influence
the bulk fluid density nor the bulk fluid viscosity. Thus, they do not interact
with (influence) the bulk fluid, but are passively transported by it;
hence the name passive scalars. The new physics introduced by $\chi_i$
is mixing, with new dimensionless parameters, the species Schmidt numbers
$\mu_0/\mu_i$, and the associated length scales for diffusion, the Batchelor
scales. See Monin-Yaglom \cite{MY} for further information.

\smallskip
Similarly to Theorem 1.1, for each fixed $\mmu=(\mu_0,\mu_1,\cdots, \mu_I),
\mu_j>0, j=0,1,\cdots, I$, there exist a global periodic solution
$(\u^\mmu, \chi_1^\mmu, \cdots, \chi_I^\mmu)(t,\x)$ in $\x\in \R^3$ with period $\TP$
and a corresponding
pressure function $p^\mmu(t,\x)$
of the Cauchy problem  \eqref{6.1}--\eqref{6.2}
such that the equations in \eqref{6.1} hold in the sense of distributions,
and the following properties hold:
\begin{eqnarray*}
&& \u^{\mmu}\in L^2(0,T; H^1)\cap C_w([0, T]; L^2)\cap C([0,T]; L^{s_1}),\\
&&\partial_t \u^\mmu\in L^2(0,T; H^{-1})+\big(L^{s_2}(0,T; W^{-1, \frac{3s_2}{3s_2-2}})\cap L^q(0,T; L^r)\big),\\
&&p^\mmu\in L^2((0, T)\times \TP)+L^{s_2}(0,T; L^{\frac{3s_2}{3s_2-2}}),\\
&&\nabla p^\mmu\in L^2(0,T; H^{-1})+L^q(0,T; L^r),\\
&&0\le \chi_i^\mmu(t,\x)\le 1, \qquad i=1,2,\cdots, I,\\
&&\chi_i^\mmu\in L^2(0,T; H^{1})
\cap C([0,T]; L^{s_2}),\\
&&\partial_t\chi_i^\mmu \in L^2(0,T; H^{-1})
+ (L^{s_2}(0,T; W^{-1, \frac{6s_2}{3s_2-2}})\cap L^q(0,T; L^{2r})\big),
\end{eqnarray*}
where $1\le s_1<2$, $1\le s_2<\infty$, $1\le q<2$, and $r=\frac{3q}{2(2q-1)}$;
and, in addition,
\begin{eqnarray*}
&&\partial_t(\frac{1}{2}|\u^\mmu|^2)+\nabla\cdot(\u^\mmu (\frac{1}{2}|\u^\mmu|^2+p^\mmu))
+\mu_0 |\nabla\u^\mmu|^2-\mu_0 \Delta(\frac{|\u^\mmu|^2}{2})\le \f\cdot\u^\mu, \label{energy-ineq:a}\\
&&\partial_t\beta(\chi_i^\mmu)+\nabla\cdot (\u^\mmu\beta(\chi_i^\mmu))
\le \mu_i\Delta \beta(\chi_i^\mmu)
\end{eqnarray*}
in the sense of distributions, for any $C^2$--function $\beta(\chi), \beta''(\chi)\ge 0$.

\begin{theorem}\label{thm:6.1}
The pathwise Kolmogorov hypothesis, {\rm Assumption (K41w)}, implies the strong compactness in
$L^2\cap L^q$ of the velocity sequence $\u^\mmu(t,\x)$, for some $q>2$,  when
the viscosity $\mmu$ tends to zero. That is,
there exist a function $\u\in L^2\cap L^q([0,T)\times\TP)$
and a subsequence (still denoted) $\u^\mmu(t,\x)$ such that
$$
\u^\mmu (t, \x)\to \u(t,\x) \qquad \mbox{in}\,\,
L^2\cap L^q([0,T)\times\TP)\,\, \mbox{as}\,\,\mmu\to 0.
$$
Furthermore,
$$
0\le \chi_i^\mmu(\x)\le 1, \qquad i=1, \cdots, I,
$$
which implies the weak-star convergence subsequentially (still denoted)
$\chi_i^\mmu(t,\x)$ in $L^\infty$ to some functions $\chi_i(t,\x),
0\le \chi_i(t,\x)\le 1$, as $\mmu\to 0$.
Moreover, the limit function $(\u, \chi_1, \cdots, \chi_I)(t,\x)$
is a weak solution of \eqref{6.1}--\eqref{6.2} and in addition,
\begin{equation} \label{6.8}
\int_{\TP}|\u(t,\x)|^2d\x \le \int_{\TP}|\u_0(\x)|^2 d\x
+2\int_0^t\int_{\TP} \u(s,\x)\cdot \f(s,\x)\, d\x ds.
\end{equation}
Denoting the weak-star limit in terms of Young measures,
with $\nu_{t,\x}^i$ the Young measure corresponding to the uniform bounded sequence
$\chi_i^\mmu$ for each $i=1,\cdots, I$,
we have
$$
\overline{\beta(\chi_i)}=\langle \nu_{t,\x}^i, \beta(\chi_i)\rangle
=w^*-\lim_{\mmu\to 0}\beta(\chi_i^\mmu)
$$
for any $C^2$--function $\beta{\chi}$. If, in addition,
$\beta''(\chi)\ge 0$, then
\begin{equation}\label{6.9}
\partial_t \overline{\beta(\chi_i)}
 +\nabla\cdot (\overline{\beta(\chi_i)}\u)\le 0
\end{equation}
in the sense of distributions.
\end{theorem}

\begin{proof} The strong convergence of $\u^\mmu(t,\x)$ can be obtained by following
the same argument as for Theorem 5.1, which implies
$$
\u^\mmu(t,\x)\to \u(t,\x) \qquad\mbox{in}\,\, L^2\cap L^q([0, T)\times\TP),  q>2.
$$

\medskip
Since $0\le \chi_i^\mmu\le 1$, there exists a further subsequence (still denoted)
$\chi_i^\mmu$ and an $L^\infty$--function $\chi_i$ with $0\le \chi_i\le 1$ such that
$$
\chi_i^\mmu \, \overset{*}{\rightharpoonup}\, \chi_i \qquad\mbox{as}\,\,\mmu\to 0.
$$
Then the strong convergence of $\u^\mmu(t,\x)$ implies that
$$
\chi_i^\mmu\u^\mmu(t,\x)\, {\rightharpoonup}\, \chi_i \u(t,\x)
$$
in the sense of distributions as $\mmu\to 0$.

This implies
$$
\partial_t\chi_i + \nabla\cdot (\chi_i\u)=0, \qquad i=1,\cdots, I,
$$
in the sense of distributions.

\medskip
Furthermore, for any $C^2$--function $\beta(\chi)$,
$\beta(\chi_i^\mmu)$ is uniformly bounded.
By the representation theorem of weak limit by Young measures
(cf. \cite{Tartar}; also see \cite{Ball}), we conclude
 $$
\beta(\chi_i^\mmu)\, \overset{*}{\rightharpoonup} \, \langle\nu^i_{t,\x}, \beta(\chi_i)\rangle
\qquad\mbox{as}\,\,\mmu\to 0,
$$
and, using the strong convergence of $\u^\mmu(t,\x)$, we further have
$$
\u^\mmu(t,\x)\beta(\chi_i^\mmu)\, {\rightharpoonup}\,
\u(t,\x)\langle\nu^i_{t,\x}, \beta(\chi_i)\rangle
$$
in the sense of distributions as $\mmu\to 0$.

We now assume
$\beta''(\chi)\ge 0$.  Since
$$
\partial_t\beta(\chi_i^\mmu)+\nabla\cdot (\u^\mmu\beta(\chi_i^\mmu))
\le \mu_i\Delta \beta(\chi_i^\mmu),
$$
we take the limit $\mmu\to 0$ above to conclude
$$
\partial_t\langle\nu^i_{t,\x}, \beta(\chi_i)\rangle
+\nabla\cdot (\u \langle\nu^i_{t,\x}, \beta(\chi_i)\rangle)
\le 0
$$
in the sense of distributions.
\end{proof}

\section{Implications to Numerical Computations}

We discuss what can be expected in terms of numerical convergence,
on the basis of the convergence framework developed in this paper.
Because the framework depends on compactness and subsequences, there
is an implicit hint that uniqueness will be a problem. In fact,
nonunique solutions for the Euler equations has been known for
over a decade; also see Scheffer \cite{Sch93} and
De Lellis-Sz\'{e}kelyhidi \cite{DeLS1,DeLS2}.
For numerical solutions, this means that LES
solutions could depend in an essential manner on the SGS models
or the equivalent in the form of modified numerical algorithms.
Stated more directly, the turbulent Schmidt and Prandtl numbers
of the models or the numerical Schmidt and Prantdl numbers of the
algorithms could influence the selection of the numerical LES solution.
Distinct but apparently converged solutions from distinct solutions
for the identical problem are thus a possibility. Moreover, the
LES solution regime is very sensitive to parameters (other than
turbulent or numerical fluid transport) which introduce a regularizing
length scale, and to the dimensionless ratios of these length
scales. The literature for multiphase flow, turbulent mixing and
turbulent chemistry has many dimensionless parameters. Those that
are sensitive (experimentally, for example) in a particular problem
can be expected to be also sensitive in the definition of the
SGS models and in their numerical manifestations as numerical
analogues of the SGS models. Our optimistic hope is that a well
designed LES algorithm should employ the SGS terms which eliminate
further nonuniqueness due to numerical or algorithmic issues, and
thus should yield a unique solution.

Assuming uniqueness of solutions has been achieved, within an
LES regime, what does convergence mean? Convergence of the velocity is
strong (norm) in an $L^q$ space. This conventional notion
of convergence requires no further discussion. For the passive scalars, the
convergence is weak-star, meaning that convergence holds for
any test function (observable) in the dual space $(L^\infty)^*$.
This dual space is a space of signed measures. The
measures and the abstract space on which they
are defined can be realized concretely in the form of
the space-time dependent Young's measures. In language more familiar
to physicists and applied mathematicians, the Young measures represent
the space-time dependent probability distribution function (pdf)
of limiting values of the $\chi_i$. Although the limit also yields
a classical weak solution, we expect that the pdf description will
prove to be more useful. The reason for this claim is that nonlinear functions
of the solution can pass to the limit in the pdf--Young measure formulation
but not in the weak star limit formulation. The classical weak
solution is the mean value of the Young measure solution, both regarded
as space-time dependent distributions. The classical weak
solution has far less information than the Young measure.
Thus we state that the
limiting concentrations $\chi_i$ are pdfs. Clearly, this is a modification
of standard ways of judging convergence, but now with a basis in mathematical
theory.

Is the Young measure limit a result of a weak mathematical technology, to be
improved at some future date? Within the generality of Theorem 6.1, this seems
unlikely. That theorem allows arbitrary and even infinite Schmidt numbers,
meaning that sharp concentration discontinuities can persist for all time.
Reasoning theoretically, the essence of the Kolmogorov theory is that
smaller and smaller vortices persist at all length scales, in the absence
of viscosity. Thus the sharp interfaces must become ever more convoluted.
Numerical evidence \cite{LimYuGli07,LimYuGli08} supports this view,
in which unregularized simulations
display turbulent mixing concentration interfaces proportional to the
available mesh degrees of freedom, uniformly under mesh refinement.
More useful is to ask what can be expected in the case of finite, bounded
Schmidt numbers. According to an analysis in \cite{MY}, the
species diffusion has a logarithmetic singularity at high Schmidt number.
But this argument omits time scales. A theoretical model \cite{LimYuGli08},
which does include time scales as well as SGS terms to modify the equations,
suggests a rapid or instantaneous smoothing of
species concentration discontinuities, but at length scales that still
may lead to Young measure solutions, i.e., convergence to a pdf.
In view of the importance of the nature of this convergence,
further studies, both numerical and theoretical, are called for.

Finally, we note that our convergence framework stops at the level
of passive scalar fields. Compressible Euler existence theories beyond this level
are known only at the level of measure-valued solutions and even
there at the level of convergent subsequences, with no proof that
the limit is actually a solution of the original equations. Because
the concentrations are already described by this measure-valued framework,
it is plausible that
such a radical change in the convergence framework is required to
move from passive scalars to true multiphase flow.
Beyond the physical issues involved in the
extension of \textit{Assumption (K41w)} to variable and species dependent
viscosity and density, the mathematical issue is the proof of uniform
estimates for partial differential equations whose coefficients are converging
to Young measures (pdfs).

\bigskip
\smallskip
{\bf Acknowledgments.}
The authors thank Peter Constantin for helpful discussions and suggestions.
Gui-Qiang Chen's research was supported in
part by the National Science Foundation under Grants DMS-0935967,
DMS-0807551, the Natural Science Foundation of China under Grant NSFC-10728101,
the Royal Society-Wolfson
Research Merit Award (UK), and the UK EPSRC Science and Innovation
Award to the Oxford Centre for Nonlinear PDE (EP/E035027/1).
James Glimm's research was supported in part by
the U.S. Department of Energy
grants DE-FC02-06-ER25770, DE-FG07-07ID14889, DE-FC52-08NA28614,
and DE-AC07-05ID14517 and
by the Army Research Organization grant
W911NF0910306.
This manuscript has been co-authored by Brookhaven Science
Associates, LLC, under Contract No. DE-AC02-98CH1-886 with
the U.S. Department of Energy. The United States Government retains,
and the publisher, by accepting this article for publication,
acknowledges, a world-wide license to publish or reproduce the published
form of this manuscript, or allow others to do so, for the United States
Government purposes.

\end{document}